\documentclass{amsart}
\usepackage{amsfonts}
\usepackage{amssymb}
\usepackage{amsthm}
\newtheorem{introthm}{Theorem}
\newtheorem{theorem}{Theorem}[section]
\newtheorem{lemma}[theorem]{Lemma}

\newtheorem{corollary}[theorem]{Corollary}

\theoremstyle{definition}
\newtheorem{example}[theorem]{Example}

\theoremstyle{remark}

\numberwithin{equation}{section}

\newcommand{\gen}[1]{\langle#1\rangle}

\newcommand{\lra}{\leftrightarrow}

\begin{document}

\title[A graph for groups]{Generalizing the enhanced power graph of a group with respect to automorphisms}

%Information for first author
\author[Mohammadian]{Abbas Mohammadian}
\address{Department of Pure Mathematics, Ferdowsi University of Mashhad, \linebreak \\ P.O.Box 1159-91775, Mashhad, Iran.\\}
\email{abbasmohammadian1248@gmail.com}

\author[Guloglu]{Ismail Guloglu}
\address{Department of Mathematics, Dogus¸ University, Istanbul, Turkey.\\}
\email{iguloglu@dogus.edu.tr}

\author[Erfanian]{Ahmad Erfanian}
\address{Department of Pure Mathematics, Ferdowsi University of Mashhad, \linebreak \\ P.O.Box 1159-91775, Mashhad, Iran.\\}
\email{erfanian@math.um.ac.ir}

\author[Lewis]{Mark L. Lewis}
\address{Department of Mathematical Sciences, Kent State University, Kent, OH  44242, U.S.A.}
\email{lewis@math.kent.edu}

%    General info

\date{February 19, 2025}



\begin{abstract}
We generalize the enhanced power graph by replacing elements with classes under automorphisms.  We show that the connectivity and diameter of this graph is similar to that of the enhanced power graph.  We consider the universal vertices of this graph and when this graph is a complete graph.  Finally, we classify when this graph is the empty graph.  
\end{abstract}

\maketitle

\section{Introduction}

Throughout this paper, $G$ is a finite group.  In recent years, there have been a great deal of interest in studying the interaction between groups and graphs.  We recommend the reader consult the excellent expository paper by Cameron \cite{camex} to see body of work.

Perhaps the graph that has received the most attention is the commuting graph.  When $G$ is a group, the {\it commuting graph} of $G$ is the graph whose vertex set is $G \setminus \{ 1 \}$ and two elements $x, y \in G \setminus \{ 1 \}$ form an edge when $xy = yx$.  Herzog, Longobardi and Maj proposed studying a modification this graph in \cite{mh-pl-mm} and defined the \textit{commuting conjugacy class graph} (or CCC-graph) of a group $G$ as the graph whose vertex set is the set of nontrivial conjugacy classes of $G$ where two distinct vertices $x^{G}$ and $y^{G}$ are adjacent if $\langle x^{\prime },y^{\prime}\rangle$ is abelian for some $x^{\prime }\in x^{G}$ and $y^{\prime } \in y^{G}$. In \cite{mh-pl-mm}, they determine the connectivity of the CCC-graph of a group $G$ and give upper bounds for the diameter of the corresponding connected components of this graph.

In fact, one can generalize this graph even further by considering automorphisms.  Let $G$ be a group and take ${\rm Aut} (G)$ to be the automorphism group of $G$.  Given a subgroup $A \le {\rm Aut} (G)$ and an element $g \in G$, define $g^A = \{g^a \mid a \in A \}$ to be the {\it $A$-orbit containing $g$}.   Notice that when $A = \{ 1 \}$, then $g^A = \{ g \}$, and hen $A = {\rm Inn} (G)$, we see that $g^A = (g^G) = \{g^x \mid x \in G \}$ is the conjugacy class of $G$.  In \cite{I.G-G.E}, the authors consider a generalization of the commuting graph by considering the graph whose vertex set is $\{ g^A : 1 \ne g \in G \}$ and given distinct $x^A, y^A$, there is an edge between $x^A$ and $y^A$ if there exists $x' \in x^A$ and $y' \in y^A$ so that $x'$ and $y'$ commute.  In that paper, the authors consider when this graph has a universal vertex and when this graph is the empty graph.  When $G$ is solvable they study the connectivity of this graph and its diameter.  They also consider when this graph is triangle-free.  This graph is studied further in \cite{few edges} where they consider when this graph is connected and it contains at most one vertex whose degree is not less than three.  

In this paper, we wish to consider similar modifications to a graph that was originally studied under the name cyclic graph (see \cite{imper} and \cite{imle}), but in most papers it is called the enhanced power graph (see \cite{AACNS}, \cite{BeBhu}, \cite{BeKiMu}, and \cite{BeDe}).  We note that there is an expository paper about results on enhanced power graphs available at \cite{enhex}.  The {\it enhanced power graph} (or cyclic graph) of $G$ is defined to be the graph whose vertex set is $G \setminus \{ 1 \}$ and there is an edge between $x,y \in G \setminus \{ 1 \}$ if $\langle x, y \rangle$ is cyclic.  Following the idea of the CCC-graph, we define the {\it cylic conjugacy class graph} of $G$ to be the graph whose vertex set is the set of nontrivial conjugacy classes of $G$ and there is an edge between distinct $x^G$ and $y^G$ if there exist $x' \in x^G$ and $y' \in y^G$ so that $\langle x', y' \rangle$ is cyclic.
%(Add here something about the enhanced graph or cyclic graph. Look at the referances and omit things which are not necessary and add things which are necessary and do not appear in the previously prepared list of references.)

With this in mind, we now define the {\it automorphic cylic graph} of $G$.  We define $\Delta (G,A)$ to be the (simple) graph whose vertex set is $\{g^{A} : 1\neq g \in G \}$.  Given distinct $x^A, y^A$, there is an edge between $x^A$ and $y^A$ if there exists $x' \in x^A$ and $y' \in y^A$ so that $\langle x', y' \rangle$ is cyclic.  Observe that when $A = 1$, then $\Delta (G,1)$ is just the enhanced power graph (i.e., cyclic graph) of $G$, and when $A = {\rm Inn} (G)$, then $\Delta (G, {\rm Inn} (G))$ is the CCC-graph of $G$.  

In Section \ref{section 2}, we show that the connectivity and diameter of $\Delta (G,A)$ is similar to the connectivity of the enhanced power graph of $G$.  Here are a sample of the results we obtain:
 
\begin{introthm} \label{intro-nonpgroupcenter} 
If $G$ is a group with $A \le {\rm Aut} (G)$ and $Z (G)$ is not a $p$-group for any prime $p$, then $\Delta (G,A)$ is connected, and moreover, $\mathrm{diam} (\Delta (G,A)) \leq 4$.  
\end{introthm} 
 
When $G$ is a $p$-group, we obtain the following.

\begin{introthm} \label{intro-pgroupcomponents} 
If $G$ is a $p$-group with $A \le {\rm Aut} (G)$, where $\mathcal{C}$ is the set of all minimal subgroups, that is the set of subgroups of $G$ of order $p$ and $\mathcal{C}_{0}$ is a complete set of representatives of the $A$-orbits on $\mathcal{C}$, then there exists a one-to-one correspondence between the connected components of $\Delta (G,A)$ and $\mathcal{C}_{0}$.  Furthermore, the vertices in $\mathcal{C}_{0}$ are adjacent to all of the vertices in its connected component.
\end{introthm} 

We are able to characterize these graphs for nilpotent groups.

\begin{introthm} \label{intro-nilpotent}
Let $G$ be a finite nilpotent group with $A \le {\rm Aut} (G)$.
	
\begin{itemize}
\item[(1)] 
If $G$ is a $p$-group, then the number of connected components of $\Delta (G,A)$ is the same as the number $A$-orbits on the subgroups of order $p$ of $G$. In particular, $\Delta (G,A)$ is connected if and only if $G$ is a homocyclic group, a quaternion group, or contains more than one involution and is a $2$-automorphic $2$-group. In any case, if $\Delta (G,A)$ is connected, then $\mathrm{diam}(\Delta (G,A)) \leq 2$.
		
\item[(2)] 
If $G$ is not a $p$-group, then $\Delta (G,A)$ is connected and $\mathrm{diam}(\Delta (G,A)) \leq 3$. Furthermore, if each of the Sylow subgroups of $G$ has only one $A$-orbit of minimal subgroups, then $\mathrm{diam} (\Delta (G,A))\leq 2$.
\end{itemize}
\end{introthm}
 
%\begin{equation*}
%E=\{(X,Y)\in \binom{V}{2}:\text{There exist } x \in X \text{ and }y \in Y\text{%
%such that}\left\langle x,y\right\rangle \text{ is cyclic}\}.
%\end{equation*}

%We shall denote the graph $\Delta (G,A)$ simply by $\Delta $ if there is no reason for confusion.  Note that $\Delta (G,1)$ is the Enhanced power graph for $G$ (see (reference)).  (It has also been studied under the name Cyclic graph in (reference)).  Also, we will consider the graph $\Delta (G,{\rm Out} (G))$ under the name {\it Cyclic conjugacy class graph} in (reference). 

We will say that a vertex is a {\it universal vertex} if it is adjacent to all the other vertices in the graph.  In Section \ref{section 3}, we will discuss the universal vertices of $\Delta (G,A)$.  We will see that when $A \le {\rm Inn}(G)$, then the universal vertices for $\Delta (G,A)$ have a similar characterization as for the enhanced power graph.

\begin{introthm}\label{intro-universal p}
If $G$ is a group and $A \leq {\rm Inn} (G)$, then $\Delta (G,A)$ has a universal vertex if and only if there exists a prime $p$ dividing $|Z(G)|$ and $G$ has a cyclic or generalized quaternion Sylow $p$-subgroup.
\end{introthm}

We are also able to characterize when $\Delta (G,A)$ will be a complete graph.  When $A \le {\rm Inn} (G)$, we get the same answer as for the enhanced power graph.

\begin{introthm}\label{intro-complete for Inn}
Let $G$ be a group.  If $A \leq {\rm Inn} (G)$, then $\Delta (G,A)$ is a complete graph if and only if $G$ is cyclic group.
\end{introthm}

For a general $A$, we get a more interesting answer.

\begin{introthm} \label{intro-complete for general A}
Let $G$ be a group and let $A \le {\rm Aut} (G)$.   Then $\Delta (G,A)$ is a complete graph if and only if $G$ is nilpotent and every element $x$ of $G$ having $p$-power order is $A$-conjugate to all elements of order $o(x)$ in $G$ for every prime $p$ dividing $|G|$.
\end{introthm}

Recall that an {\it empty graph} is a graph with no edges.  In Section \ref{section 4}, we consider the question of when $\Delta (G,A)$ is an empty graph.  We prove the following. 

\begin{introthm} \label{intro-NoEdges} 
Let $G$ be a group with $A \le {\rm Aut} (G)$.  Then $\Delta (G,A)$ is the empty graph if and only if every nonidentity element $x \in G$ has prime order, $N_{A}(\left\langle x\right\rangle )$ acts transitively on $\left\langle x \right\rangle \setminus \{ 1 \}$ for every $1 \neq x \in G$, and one of the following holds:
\begin{enumerate}
\item $G$ is a $p$-group for some prime $p$ and is of exponent $p$, 
\item $G$ is a Frobenius group with Frobenius kernel of prime exponent and a Frobenius complement of prime order, 
\item $G$ is isomorphic to $A_{5}$ and if $A \leq {\rm Aut} (A_{5})$, then either $A \cong S_{4}$ or $S_{5}$.
\end{enumerate}
\end{introthm}

For the enhanced power graph and the cyclic conjugacy class graph, we obtain the following.

\begin{introthm}\label{intro-last}
Let $A \leq {\rm Inn}(G)$.  Then $\Delta (G,A)$ is the empty graph if and only if either $G$ is an elementary abelian $2$-group or is a Frobenius group with Frobenius kernel that is an elementary abelian $3$-group and Frobenius complement of order $2$.
\end{introthm}

%---------------------------------------------------------------------------
\section{{\protect\Large On the connectivity of }$\Delta (G,A)$}\label{section 2}

In this section we shall discuss some issues related to connectivity of the graph $\Delta = \Delta (G,A)$.  The question of connectivity of this graph is of course closely related to the connectivity of $\Delta (G,1)$ and the Gruenberg-Kegel graph of $G$.  As such, we will see that the results in this sections are generalizations of similar results for the enhanced power graphs, and thus, many of the proofs are similar to the proofs of results in \cite{direct}.  This next Lemma is essentially Lemma 2.1 of \cite{direct} and the proof is similar enough that we do not repeat it here.

\begin{lemma} \label{coprimeelements} 
If $G$ is a group and $A \le {\rm Aut} (G)$ and $x, y \in G\setminus \{1\}$ such that $xy = yx$ and $(|x|,|y|)=1$, then $x^{A}$ is adjacent to $y^{A}$ in $\Delta (G,A)$.
\end{lemma}

Compare the following Lemma with Theorem C of \cite{direct}.

\begin{lemma} \label{connection 1} 
If $G$ is a group with $A \le {\rm Aut} (G)$ and $x, y \in G \setminus \{1\}$ are elements of prime order $p$ such that $x^{A}$ and $y^{A}$ lie in the same connected component of $\Delta (G,A)$.  Then either $C_{G}(x)$ and $C_{G}(y)$ are not $p$-groups or $\left\langle x \right\rangle = \left\langle y\right\rangle ^{a}$ for some $a\in A$.  In particular, $\Delta (G,A)$ is disconnected if one of the following holds:

\begin{itemize}
\item  $G$ is not a $p$-group and there exists an element $x\in G$ with order $p$, such that $C_{G}(x)$ is a $p$-group, or

\item  $G$ is a $p$-group and $A$ does not act transitively on the set of subgroups of order $p$ of $G$.
\end{itemize}
\end{lemma}

\begin{proof}
Suppose that $x, y \in G\setminus \{ 1 \}$ are elements of prime order $p$ and suppose that there is a path $X_{0} \leftrightarrow X_{1} \leftrightarrow \cdots \leftrightarrow X_{n} \leftrightarrow X_{n+1}$ joining the vertices $x^{A} = X_{0}$ and $y^{A} = X_{n+1}$.  Then there exists representatives of these $A$-orbits, say $x_{i}\in X_{i}$ such that $x_{0} = x$, $\left\langle x_{i}, x_{i+1} \right\rangle $ is cyclic, and $x_{n+1} = y^{a}$ for some $a \in A$, with $i = 0, 1, \ldots, n, n+1$.  If possible, let $k$ be the smallest integer such that $x_{k}$ has order divisible by some prime $q \neq p$.  Then, for $i$ with $i+1 < k$ we see that $\left\langle x_{i}, x_{i+1} \right\rangle $ is a cyclic $p$-group, and hence, there is a unique subgroup $Z$ of order $p$ which has to be a subgroup of $\left\langle x_{i} \right\rangle \cap \left\langle x_{i+1}\right\rangle$.  This implies that $Z = \left\langle x \right\rangle$ and that $x_{k} \in C_{G}(x)$.  Therefore, if $C_{G}(x)$ is a $p$-group, then $Z = \left\langle x \right\rangle = \left\langle y^{a}\right\rangle $ for some $a\in A$.  This proves the claim.
\end{proof}

Compare the following with Lemma 3.1 of \cite{direct}.

%\begin{theorem} \label{nonpgroupcenter} 
%If $G$ is a group with $A \le {\rm Aut} (G)$ and $Z (G)$ is not a $p$-group for any prime $p$, then $\Delta (G,A)$ is %connected, and moreover, $\mathrm{diam} (\Delta (G,A)) \leq 4$.  
%\end{theorem}

\begin{proof}[Proof of Theorem \ref{intro-nonpgroupcenter}]
Let $z_{p},z_{q}\in Z(G)$ be elements of order $p$ and $q$, respectively, where $p$ and $q$ are distinct primes. Also, let $x, y \in G$ be arbitrary elements with $|x| = m$ and $|y| = n$ for positive integers $m$ and $n$ greater than $1$.

Suppose there exist a prime $r \in \{p,q\}$ so that $r$ does not divide $m$ and $n$, then $x^{A} \leftrightarrow (z_{r})^{A} \leftrightarrow y^{A}$, and hence, $d (x^{A}, y^{A}) \leq 2$.  Thus, we may assume that each prime in $\{ p, q \}$ divides at least one of the members of $\{ m, n \}$.  Without loss of generality, we may assume that $p \mid m$.  If $q$ does not divide $n$, then we obtain $x^{A} \leftrightarrow (x^{\frac{m}{p}})^{A} \leftrightarrow (z_{q})^{G} \leftrightarrow y^{G}$ and $d (x^{A}, y^{A}) \leq 3$.  So, we may assume $q$ divides $n$.  Changing the roles of $p$ and $q$, we similarly  see that we also may assume that $p \mid n$ and $q \mid m$.  Thus, we have $x^{A} \leftrightarrow (x^{\frac{m}{p}})^{A} \leftrightarrow (z_{q})^{A} \leftrightarrow (y^{\frac{n}{p}})^{A} \leftrightarrow y^{A}$, and thus, $d (x^{A}, y^{A}) \leq 4$. Therefore, we conclude that $\mathrm{diam} (\Delta (G,A)) \leq 4$.
\end{proof}

We note that the bound in Theorem \ref{intro-nonpgroupcenter} is sharp.  Consider the group $G = Z \times FH$ where $Z \cong H \cong \mathbb{Z}_{6}$ and $F \cong \mathbb{Z}_{7}$, so that $FH$ is a Frobenius group of order $42$.  Let $H = \left\langle x\right\rangle$ and $F= \left\langle u \right\rangle $ and $y = x^{u}$.  If we take $A=1$, then we see that $d (x,y) = 4$ in $\Delta (G,1)$, although $d(x^{Inn(G)},y^{Inn(G)}) = 0$.  
%TCIMACRO{\U{2124} }%
%BeginExpansion

%EndExpansion

%TCIMACRO{\U{2124} }%
%BeginExpansion

%EndExpansion

We now consider the connected components of $\Delta (G,A)$ when $G$ is a $p$-group.  One should compare this theorem with Lemma 2.3 of \cite{direct}.

%\begin{theorem} \label{pgroupcomponents} 
%If $G$ is a $p$-group with $A \le {\rm Aut} (G)$, where $\mathcal{C}$ is the set of all minimal subgroups, that is the set of subgroups of $G$ of order $p$ and $\mathcal{C}_{0}$ is a complete set of representatives of the $A$-orbits on $\mathcal{C}$, then there exists a one-to-one correspondence between the connected components of $\Delta (G,A)$ and $\mathcal{C}_{0}$.  Furthermore, the vertices in $\mathcal{C}_{0}$ are adjacent to all of the vertices in its connected component.
%\end{theorem}

\begin{proof} [Proof of Theorem \ref{intro-pgroupcomponents}]
We want to prove that for any connected component $O$ of $\Delta (G,A)$ here exists a unique element $S$ of $\mathcal{C}_{0}$ such that 
$$S \setminus \{1\} \subset \widetilde{O} := \bigcup \nolimits_{Z \in O} Z.$$
Let $x^{A} \in O$ for some $1 \neq x \in G$. Then $y^{A}$ is adjacent to $x^{A}$ for any $1 \neq y \in \langle x \rangle$. In particular, if $X$ is the unique subgroup of order $p$ in $\langle x\rangle$,  we see that $X \setminus \{1\} \subset \widetilde{O}$.   Hence, there exists an element $S$ in $\mathcal{C}_{0}$ and an automorphism $a \in A$ such that $X^{a}=S$.  As $O^{a}=O$, we see that for each connected component $O$ of $\Delta (G,A)$, there exists a subgroup $S$ of $G$ of order $p$ in $\mathcal{C}_{0}$ such that $S \setminus \{ 1 \} \subset \widetilde{ O}$.  In order to prove the claim, we need to prove that this set $S$ is uniquely determined by $O$. Let $P = \langle x \rangle$ and $Q = \langle y \rangle$ two distinct subgroups of $G$ of order $p$ such that $x^{A}$ and $y^{A}$ lie in the same connected component $O$. Then there exists an automorphism $a \in A$ and a sequence $x_{0}=x, x_{1}, x_{2}, \dots, x_{n}, x_{n+1}=y^{a}$ of nonidentity elements of $G$ such that $\langle x_{i}, x_{i+1} \rangle$ is cyclic for each $i = 0, 1, \dots, n$.  Since $\langle x_{i}, x_{i+1} \rangle$ is a nontrivial cyclic $p$-group, it contains a unique subgroup $Z$ of order $p$. But then $Z \leq \langle x_{i} \rangle$ and $Z \leq \langle x_{i+1} \rangle $ for every $i = 0, 1, \dots, n$ implying that $\langle x_{0}\rangle = Z = \langle x_{n+1}\rangle$, that is $\langle x \rangle =\langle y^{a}\rangle$ and hence, $P = Q^{a}$ for some $a \in A $. This completes the proof.
\end{proof}

%Compare this theorem to

%\begin{theorem} \label{connectedcomponentpgroup} 
%Suppose $G$ is a $p$-group for some prime $p$ and $A \le {\rm Aut} (G)$.  If $O$ is a connected component of $\Delta (G,A)$, then the vertices $x^{A}\in O$ corresponding to minimal subgroups $\left\langle x\right\rangle$ of $G$ are adjacent to every vertex of $O$.
%\end{theorem}

%\begin{proof}
%Aiming for a contradiction, suppose that $x^{A}$ is not adjacent to every vertex in $O$. Then there exist elements $u, y \in G$ such that $u^{A}, y^{A}\in O$, so that $y^{A}$ is not adjacent to $x^{A}$ and $x^{A}\leftrightarrow u^{A} \leftrightarrow y^{A}$.  Furthermore, we may assume without loss of generality that the subgroup $\langle x, u \rangle $ is cyclic and therefore, it has a unique subgroup of order $p$, namely $\langle x \rangle$. Note that $\langle u \rangle$ also has a unique subgroup of order $p$.  It follows that $\langle x \rangle \leq \langle u\rangle$ as $\langle u\rangle \leq \langle x,u\rangle$.  Now, $\langle u,y^{a}\rangle $ is cyclic for some automorphism $a\in A$, and $x \in \left\langle u,y^{a} \right\rangle$.  Therefore, the subgroup $\langle x, y^{a}\rangle$ is also cyclic which yields that $x^{A}$ is adjacent to $y^{A}$, a contradiction.
%\end{proof}

\begin{example}
\textbf{\ }
Suppose $G$ is a $p$-group for some prime $p$.  We note that the connected components of $\Delta (G,A)$ need not be cliques, that is, complete graphs.  To see this, consider the group $G$ of order $27$, as follows:
$$ G = \gen{ a,b \mid a^{9} = b^{3} = 1, [a,b] = a^{3} }.$$
Let $\tau$ be the automorphism of $G$ such that $a^{\tau }=a^{-1}$ and $b^{\tau }=b$. Assume that $A = {\rm Inn(G) } \gen{\tau}$. Then $\Delta (G,A)$ has vertices represented by the elements: $a^{3},b,b^{2},a,ba,b^{2}a,$ and is not connected; it has two connected components which consist of vertices represented by $b,b^{2}$ and by $a^{3},a,ba,b^{2}a$. The first connected component is a clique, but the second connected component is not, the vertex represented by $a$ is not adjacent to other vertices represented by elements of order nine.
\end{example}

%\begin{remark}
The groups $G$ that admit a group of automorphisms which acts transitively on the set of minimal subgroups of $G$ are either (1) abelian $p$-groups for some prime $p$ that are homocyclic or (2) nonabelian $2$-groups which are called $2$-automorphic $2$-groups.  These are $2$-groups $G$ of exponent $4$ and nilpotency class $2$ so that $G^{\prime} = \Phi (G) = Z(G) = \Omega_{1}(Z(G))$ and $\left\vert G \right\vert = \left\vert Z(G) \right\vert^{2}$ or $\left\vert G\right\vert = \left\vert Z(G) \right\vert ^{3},$ with $\Omega _{1}(Z(G)) \setminus \{1\}$ being the set of involutions of $G$ (See \cite{2-automorphic}).  It is has recently been proved in \cite{gross} that these are the Suzuki $2$-groups.  This next theorem should be compared with Corollary 4.2 of \cite{direct}.
%\end{remark}

%\begin{theorem} \label{nilpotent}
%Let $G$ be a finite nilpotent group with $A \le {\rm Aut} (G)$.

%\begin{itemize}
%\item[(1)] 
%If $G$ is a $p$-group, then the number of connected components of $\Delta (G,A)$ is the same as the number $A$-orbits on the subgroups of order $p$ of $G$. In particular, $\Delta (G,A)$ is connected if and only if $G$ is a homocyclic group, a quaternion group, or contains more than one involution and is a $2$-automorphic $2$-group. In any case, if $\Delta (G,A)$ is connected, then $\mathrm{diam}(\Delta (G,A)) \leq 2$.

%\item[(2)] 
%If $G$ is not a $p$-group, then $\Delta (G,A)$ is connected and $\mathrm{diam}(\Delta (G,A)) \leq 3$. Furthermore, if each of the Sylow subgroups of $G$ has only one $A$-orbit of minimal subgroups, then $\mathrm{diam} (\Delta (G,A))\leq 2$.
%\end{itemize}
%\end{theorem}

\begin{proof}[Proof of Theorem \ref{intro-nilpotent}]
We recall that in a nilpotent group, elements of coprime orders commute.

(1) By Theorems \ref{intro-pgroupcomponents} and \cite{2-automorphic}, we need only to prove that $\mathrm{diam}(\Delta (G,A))\leq 2$ when $\Delta (G,A)$ is connected.  We see that ${\rm diam} \Delta )\leq 2$ because if $x, y$ are any two elements of $G \setminus \{1\}$, then there exists integers $u$ and $v$ such that $x^{u}$ and $y^{v}$ are in $\Omega _{1}(Z(G))$ by \ref{connection 1}, and hence, there exists an automorphism $1 \ne a\in A$ such that $\left\langle (x^{u})^{a} \right\rangle = \left\langle y^{v} \right\rangle $ which yields (if necessary by changing $v$) without loss of generality that $x^{A} \leftrightarrow (x^{u})^{A} = \left\langle y^{v} \right\rangle^{A} \leftrightarrow y^{A}$.

(2) Consider $x, y \in G \setminus \{1\}$ so that $|x| = m$ and $|y| = n$.  If both $x$ and $y$ are $p$-elements for some prime $p$, then for any other prime number $q$ dividing the order of $G$, and any element $z$ of $G$ having order $q$, we obtain $x^{A} \leftrightarrow z^{A} \leftrightarrow y^{A}$ by \ref{coprimeelements}, and hence, $d (x^{A}, y^{A}) = 2$.  Thus, there exist two distinct prime numbers $p$ and $q$ such that $p \mid m$ and $q \mid n$. Then $x^{A} \leftrightarrow
(x^{\frac{m|}{p}})^{A} \leftrightarrow (y^{\frac{n}{q}})^{A} \leftrightarrow y^{A}$ by \ref{coprimeelements} and hence, $d (x^{A}, y^{A}) \leq 3$. This shows that $\Delta (G,A)$ is connected and $\mathrm{diam} (\Delta (G,A)) \leq 3$.

Suppose now that each of the Sylow subgroups of $G$ contains exactly one $A$-orbit of minimal subgroups and that a prime $p$ divides $(m,n)$. Then $x^{\frac{m}{p}}$ and $y^{\frac{n}{p}}$ are two elements of order $p$ that lie in the Sylow $p$-subgroup of $G$.  Hence, there exists an automorphism, $a\in A$ so that 
$$\left\langle x^{\frac{m}{p}}\right\rangle^{a} = \left\langle y^{\frac{n}{p}} \right\rangle.$$
This yields the path: 
\begin{equation*}
x^{A}=(x^{a})^{A}\leftrightarrow ((x^{\frac{m}{p}})^{a})^{A}=(y^{\frac{n}{p}%
})^{A}\leftrightarrow y^{A}
\end{equation*}%
which implies that $d(x^{A},y^{A})\leq 2$.  Combining this with the first paragraph, and using the observation that $d(x^{A},y^{A})\leq 1$ when $(m,n)=1$, we see that the proof is finished.
\end{proof}

We now present an example where $\Delta (G,A)$ has diameter $3$.

\begin{example}
Let $G = P \times R$ where $P = \left\langle a, b \right\rangle$ is a Klein four group and $R = \left\langle c , d \right\rangle$ an extraspecial group of order $27$ and of exponent $3$.  Let $A$ be the set of inner automorphisms of $G$ acting by conjugation on $G$.  By Theorem \ref{intro-nilpotent}, $\Delta(G,A)$ is connected.  We work to show that $\mathrm{diam} (\Delta(G,A)) = 3$.  For that purpose, consider the elements $x=ac$ and $y=bd$.  Clearly, the conjugacy classes of $x$ and $y$ (that is the $A$-orbits containing $x$ and $y$) are not adjacent.  So suppose that $d (x^{G}, y^{G}) = 2$.  Hence, there exists an element $1 \neq u \in G$, such that $x^{A} \leftrightarrow u^{A} \leftrightarrow y^{A}$.  This implies, without loss of generality, that $\langle x, u \rangle$ and $\langle u, y^{g} \rangle$ are both cyclic subgroups for some element $g \in G$. 

A subgroup $T$ of $G$ is cyclic if and only if it is abelian and all its Sylow subgroups are cyclic.  This means that both  $T\cap P$ and $T\cap R$ are cyclic.  So, $\langle x, u \rangle \cap P$ and $\langle x, u \rangle \cap R$ and similarly $\langle u, y^{g} \rangle \cap P$ and $\langle u,y^{g} \rangle \cap R$ are cyclic.  Let us write $u = tz$ for elements $t \in P$ and $z \in R$. Then we see that 
$$\left\langle t, a\right\rangle, \left\langle t, b^{g} \right\rangle = \left\langle t, b \right\rangle, \left\langle z, c \right\rangle, \left\langle z, d^{g} \right\rangle$$ 
are all cyclic.  Because $\left\langle t, a \right\rangle, \left\langle t,b\right\rangle$ are both cyclic implies that $t = 1$, and hence, $z \neq 1$ as $u\neq 1$.  Thus, the fact that  $\left\langle z, c \right\rangle, \left\langle z, d^{g} \right\rangle$ are both cyclic implies that $c, d^{g} \in \left\langle z \right\rangle$.  This is not possible because for any $g \in G$ we see that $d^{g} \in dR^{\prime }$, and hence  $\left\langle c, d^{g} \right\rangle = R$ is not cyclic.  Thus, $d (x^{G}, y^{G}) \neq 2$, and $\mathrm{diam} (\Delta (G,A)) = 3$.
\end{example}

We now introduce two useful conditions on centralizers, and we study the consequences of these conditions.  Compare this with Lemma 3.2 of \cite{direct}.

\begin{lemma} \label{centralizerisnotpgroup} 
Let $G$ be a finite group with $A \le {\rm Aut} (G)$ and assume $Z(G)$ is a non-trivial $p$-group for some prime $p$.  If $C_{G}(x)$ is not a $p$-group for any element $x$ of order $p$, then $\Delta (G,A)$ is connected with $\mathrm{ diam} (\Delta (G,A)) \leq 6$.
\end{lemma}

\begin{proof}
Consider an element $x \in G\setminus \{1\}$ and $z$ a non-trivial element of $Z(G)$.  For a suitable integer $t$, we have  $|x^{t}| = q$, where $q$ is a prime.  If $q \neq p$, then $x^{A} \leftrightarrow (x^{t})^{A} \leftrightarrow z^{A}$. 

Assume $q = p$.   That is, assume $x$ is a $p$-element.  Since $C_{G} (x^{t})$ is not a $p$-group, there exists an element  $y \in C_{G} (x^{t})$ with $|y| = r$, where $r$ is a prime different from $p$.  Now, $x^{A}\leftrightarrow  (x^{t})^{A}\leftrightarrow y^{A} \leftrightarrow z^{A}$.  Hence, $\mathrm {d} (x^{A},z^{A}) \leq 3$ for an arbitrary $x \in G\setminus \{1\}$.  Finally, if $x, y \in G \setminus \{1\}$, then a path between $x^{A}$ and $y^{A}$ of length at most six can be built by passing through $z$. The result follows.
\end{proof}

\section{Universal vertices}\label{section 3}

Next we want to look at the groups $G$ admitting a group $A$ acting by automorphisms on $G$ in such a way that the graph $\Delta (G,A)$ becomes naturally connected because it has a {\it universal vertex} (in the literature on groups with graphs this is often called a dominating vertex, but in the graph theory literature universal seems preferred), that is, a vertex which is adjacent to every other vertex.

%\begin{definition}
Let G be a group. For any $x\in G$ we denote the subset 
\begin{equation*}
{\rm Cyc}_{G} (x) = \{y \in G: \left\langle x,y \right\rangle \text{ is cyclic}\}.
\end{equation*}%
The intersection of ${\rm Cyc}_G (x)$ over all elements $x$ of $G$ is denoted by $K (G)$.  Thus, 
\begin{equation*}
K (G) = \{g \in G : \left\langle x,g \right\rangle \text{ is cyclic for all ~} x \in G\}.
\end{equation*}%

Observe that $K(G)$ is the set of universal vertices for the enhanced power graph of $G$ along with the identity.  This set was apparently first studied in \cite{cycels} (in another context) and in Theorem 1 of \cite{cycels}, $K(G)$ is proved to be  a characteristic subgroup of $G$ which is cyclic and is contained in $Z(G)$.  We next begin to consider properties of universal vertices in $\Delta (G,A)$.
%\end{definition}

\begin{theorem}  \label{dominating vertex}
Let $G$ be a group with $A \le {\rm Aut} (G)$ and consider $x \in G$.  Then the following are equivalent:
\begin{enumerate}
\item The vertex $x^{A}$ of $\Delta (G,A)$ is a universal vertex
\item \[G = \bigcup\limits_{a\in A} {\rm Cyc}_{G} (x)^{a}.\] 
\end{enumerate}
When this occurs, the following are true:
\begin{enumerate}
\item[(i)] All subgroups of $G$ of order dividing the order of $x$ are $A$-conjugate to a subgroup $\left\langle x \right\rangle$ and $G = \bigcup_{a \in A} C_G (x)^a$. %intersects nontrivially every $A$-class of elements of $G$. %whose orders are coprime to $o(x)$.
\item[(ii)] Assume in addition that the conjugacy class $x^{G}$ is $A$-invariant (which is for example the case if $A \leq Inn(G)$).  If $x^{A}$ is a universal vertex of $\Delta (G,A)$, then $x \in K (G)$.
\item[(iii)] If $x \in K (G)$, then $x^{A}$ is a universal vertex of $\Delta (G,A)$ and $\left\langle x \right\rangle$ is the only subgroup of $G$ of order equal to the order of $x$.
\end{enumerate}
\end{theorem}

\begin{proof}
If $x^{A}$ is a universal vertex, then for any element $1 \neq y$ there exists an automorphism $a\in A$ such that $\left\langle x, y^{a} \right\rangle$ is cyclic, which implies that $y \in \bigcup\limits_{a \in A} {\rm Cyc}_{G} (x)^{a}$. We conclude $G = \bigcup\limits_{a \in A} {\rm Cyc}_{G} (x)^{a}$.  Suppose $y \in G$ has order dividing the order of $x$.  Since $x^A$ and $y^A$ are adjacent, there is an $a \in A$ so that $\langle x, y^a \rangle$ is cyclic.  Now, since $\langle x, y^a \rangle$ has a unique subgroup of order $o(x)$, we conclude that $\langle x \rangle \ge \langle y^a \rangle$.  This implies that every subgroup of order dividing $o(x)$ is $A$-conjugate to a subgroup of $\langle x \rangle$.  Suppose $g \in G$.  We know that $x^A$ and $g^A$ are adjacent, so there is some automorphism $a \in A$, so that $\langle x^a, g \rangle$ is cyclic.  This implies that $u \in C_G (x^a) = C_G (x)^a$ and thus, $G \le \bigcup_{a \in A} C_G (x)^a$.  We conclude that $G = \bigcup_{a \in A} C_G (x)^a$.

Conversely, suppose $G = \bigcup\limits_{a \in A} {\rm Cyc}_{G} (x)^{a}$.  Then for any $g \in G$, there exists an automorphism $a \in A$ such that $\left\langle g, x^{a} \right\rangle$ is cyclic.  That is, $x^{A}$ is adjacent to $g^{A}$, and therefore, $x^{A}$ is a universal vertex of $\Delta (G,A)$.  %Next, suppose all subgroups of $G$ with order equal to $o(x)$ are $A$-conjugate to $\langle x \rangle$ and $C_G (x)$ intersects nontrivially all $A$-classes of elements of $G$ whose order is coprime to $o(x)$.  

Suppose $A$ leaves $x^{G}$ invariant.  That is, for any automorphism $a\in A$ there exists an element $g \in G$ such that ${x^{a} = x^{g}}$.  Since $x^{G}$ is a universal vertex of $\Delta (G,A)$, we have $G = \bigcup\limits_{a \in A} {\rm Cyc}_{G} (x)^{a} = \bigcup\limits_{g \in G} {\rm Cyc}_{G} (x)^{g}$.  Thus, writing $X = {\rm Cyc}_{G} (x) - \{ 1 \}$, we obtain  $\bigcup\limits_{g \in G} {\rm Cyc}_{G} (x)^{g} = \{1\} \cup \bigcup\limits_{g \in G} X^{g}$ and hence, $\left\vert G \right\vert \leq 1+ [G:N_{G}(X)] \left\vert X \right\vert$.  This is not possible since $X \subset N_{G} (X) - \{1\}$ if ${\rm Cyc}_{G} (x)$ is a proper subset of $G$.  So we get ${\rm Cyc}_{G} (x) = G$, and hence, $x \in K (G)$. 

Conversely, if $x$ is a nontrivial element of $K (G)$, then clearly (as $x$ is a universal vertex of $\Delta (G,1)$) $x^{A}$ is adjacent to every other vertex in $\Delta (G,A)$.  But $K (G)$ is a characteristic subgroup of $G$ which is cyclic and is contained in $Z (G)$.  Therefore, $\left\langle x \right\rangle = \left\langle x \right\rangle^{a} = \left\langle y \right\rangle$ for any automorphism $a\in A$ and for any element $y\in G$ whose order is equal to the order of $x$.
\end{proof}

%\begin{remark}
We now present an example of a group $G$ with an automorphism group $A$ so that $\Delta (G,A)$ has universal vertices where $K (G) = 1$.  Let $G$ be an elementary abelian group of order $p^{n}$ with integer $n > 1$.  Then we can consider $G$ to be the additive group of the field with $p^{n}$ elements.  Let $A$ be the multiplicative group of this field.  Then $A$ acts Frobeniusly on $G$ and is transitive on $G \setminus \{1\}$.  For any element $1 \neq x \in G$, we have ${\rm Cyc}_{G} (x) = \left\langle x \right\rangle$ and thus, $K (G) = 1$.  We deduce that $\Delta (G,A)$ has only one vertex, and in particular, is vacuously a complete graph.
%\end{remark}

We next show that if all of the vertices corresponding to $A$-classes of $p$-elements are adjacent to all of vertices of $A$-classes corresponding to $q$-elements in $\Delta (G,A)$ when $p$ and $q$ are distinct primes, then $G$ has to be nilpotent.

\begin{theorem}\label{p-element and q-element}
Let $G$ be a group with $A \le {\rm Aut} (G)$.  Suppose that for any two distinct primes $p$ and $q$ and for all elements $x, y \in G \setminus \{ 1 \}$ where $x$ is a $p$-element and $y$ is a $q$-element, we have $x^A \lra y^A$ in $\Delta (G,A)$. Then $G$ is nilpotent.
\end{theorem}

\begin{proof}
In \cite{I.G-G.E}, a more general form of this theorem is proved.
\end{proof}

We can now show that if $\Delta (G,A)$ is a complete graph, then $G$ must be nilpotent.

\begin{corollary}\label{D(G,A) is a complete graph}
Let $G$ be a group with $A \le {\rm Aut} (G)$.  If $\Delta (G,A)$ is a complete graph, then $G$ is a nilpotent group.
\end{corollary}

We next see when we restrict to the case when $A$ is a subgroup of the inner automorphism group of $G$, we deduce that $\Delta (G,A)$ is complete graph if and only if $G$ is cyclic.  This generalizes the case of $\Delta (G,1)$ which is the enhanced power graph.   %Note that this Theorem \ref{intro-complete for Inn} from the Introduction.  

%\begin{corollary}\label{complete for Inn}
%Let $G$ be a group.  If $A \leq {\rm Inn} (G)$, then $\Delta (G,A)$ is a complete graph if and only if $G$ is cyclic group.
%\end{corollary}

\begin{proof}[Proof of Theorem \ref{intro-complete for Inn}]
By	Corollary \ref{D(G,A) is a complete graph}, we know that $\Delta (G,A)$ complete implies $G$ is nilpotent.  Conversely, suppose $A \le G$ where $G$ is a nonabelian nilpotent group.  Since $A \le G$, the class $z^A$ is a universal vertex if and only if $z \in K (G)$.  Therefore, $\Delta (G,A)$ is not complete.  Thus, we may assume that $G$ is abelian.  We already know that if $G$ is not cyclic, then $\Delta (G,A)$ is not connected, let alone complete.  
\end{proof}
    
More precisely, there exist non-cyclic nilpotent groups $G$ (for example $G = D_8$, the dihedral group of order $8$) such that $\Delta (G, A)$ is not complete for any subgroup $A\leq {\rm Aut} (G)$. For example, $\Delta (D_8, {\rm Aut} (D_8))$ has $3$ vertices corresponding to: the orbit of the involution in the center of $D_8$, the orbit of the elements of order $4$, and the orbit of the involutions outside the center of $D_8$, the last two of these vertices are not adjacent.
%\end{proof}

As another corollary, we obtain the following result, which is a theorem of A. Mahmoudifar and A. Babai (see \cite{A.M-A.B}).  We note that this was also proved indepedently for the enhanced power graph by the second author with others in Theorem 3.3 of \cite{zgps}.  We note that for the enhanced power graph, additional information was found reqarding the structure of groups having a universal vertex in \cite{univ}.  We wonder whether a similar result can be obtained in relation to this graph.   %Observe that the following is Theorem \ref{intro-universal p} from the Introduction.

%\begin{corollary}\label{universal p}
%If $G$ is a group and $A \leq {\rm Inn} (G)$, then $\Delta (G,A)$ has a universal vertex if and only if there exists a prime $p$ dividing $|Z(G)|$ and $G$ has a cyclic or generalized quaternion Sylow $p$-subgroup.
%\end{corollary}

\begin{proof}[Proof of Theorem \ref{intro-universal p}]
Assume $\Delta (G,A)$ has a universal vertex.  By Theorem \ref{dominating vertex}, $G$ has a unique subgroup $\left\langle x \right\rangle$ of prime order, say $p$, this element $x$ lies in $K (G)$, so $p$ divides $|K (G)|$ which divides $|Z(G)|$.  Also, if $P$ is a Sylow $p$-subgroup of $G$, then $P$ is cyclic or generalized quaternion.

Conversely, suppose that $p$ is a prime so that a Sylow $p$-subgroup $P$ of $G$ is cyclic or generalized quaternion and $p$ divides $|Z(G)|$.  Let $\left\langle y \right\rangle$ be the unique subgroup of $P$ of order $p$.  It is conjugate, and hence, equal to the subgroup of order $p$ in $Z(G)$.  In particular, $G$ has only one subgroup of order $p$.   Consider an $A$-class $g^{A} \in V(\Delta (G,A)) \setminus \{ y^{A} \}$.  If $(|y|, |g|) = 1$, then $y^{A}$ and $g^{A}$ are adjacent in $\Delta (G,A)$, since $y$ and $g$ are commuting elements with coprime orders.  If $p$ divides $|g|$, then the uniqueness of the subgroup of order $p$ forces $\left\langle y \right\rangle \le \left\langle g \right\rangle$.  Again, $y^{A}$ is adjacent to $g^{A}$.  We conclude that $y$ is a universal vertex.
\end{proof}

We also have the following corollary.

\begin{corollary}
Suppose $G$ is a group with $A \le {\rm Aut} (G)$.  Let  $x^{A}$ be a universal vertex of $\Delta(G,A)$. If $x$ is of order $p$ and $G$ is $p$-solvable with $O_{p'}(G) = 1$, then $\Omega_{1} (Z(O_{p}(G))) = \Omega_{1} ((O_{p} (G))$ and $O_{p} (G)$ is either abelian or is a $2$-automorphic $2$-group.
\end{corollary}

\begin{proof}
By Theorem \ref{dominating vertex}, we kow that $\langle x \rangle$ is $A$-conjugate to all the subgroups of order $p$ in $G$.  Since $O_{p'} (G) = 1$, we have $O_p (G) > 1$ and thus, $Z(O_p (G)) >1$.  By Cauchy's theorem, $Z(O_p (G))$ is going to contain a subgroup of order $p$.  Since $Z(O_p (G))$ is characteristic in $G$, we conclude that $Z (O_p (G))$ contains all of the subgroups of order $p$ in $G$.  Hence, $\Omega_{1} (Z(O_{p}(G))) = \Omega_{1} ((O_{p} (G))$.  As we mentioned before, having a group of automorphisms that acts transitively on the subgroups of order $p$ implies that a Sylow $p$-subgroup of $G$ and hence $O_p (G)$ is abelian or is a $2$-automorphic $2$-group. 
\end{proof}

The question arises: is it possible to characterize the universal vertices of $\Delta (G,A)$ when $A$ is not contained in ${\rm Inn} (G)$?  In particular, we have shown that all subgroups of $G$ whose order divides $o(x)$ must be $A$-conjugate to $\langle x \rangle$ and $G = \bigcup_{a \in A} C_G (x)^a$.  Is this enough to characterize the universal vertices or are additional conditions needed?  Do the universal vertices of $\Delta (G,A)$ form a subgroup of $G$?  We note that when removing the $p$-solvable hypothesis, one will want to consider Theorem A of \cite{KNST}, where they classify the groups (in particular, the non-solvable groups) where the $p$-elements are all conjugate.  To determine if they yield a complete vertex in our graph, one needs to see if their centralizers intersects nontrivially every $A$-class of $G$. 

%Question: Can we characterize the pairs $(G,A)$ where $A$ is not contained in ${\rm Inn} (G)$ so that $\Delta (G,A)$ is a complete graph?

We now characterize the pairs $(G,A)$ where $\Delta (G,A)$ is a complete graph when $A$ is not contained in ${\rm Inn} (G)$. 

%\begin{theorem} \label{complete for general A}
%Let $G$ be a group and let $A \le {\rm Aut} (G)$.   Then $\Delta (G,A)$ is a complete graph if and only if $G$ is nilpotent and every element $x$ of $G$ having $p$-power order is $A$-conjugate to all elements of order $o(x)$ in $G$ for every prime $p$ dividing $|G|$, i.e. $G$ is either abelian or the direct product of an abelian group with a $2$-automorphic $2$-group.
%\end{theorem}

\begin{proof}[Proof of Theorem \ref{intro-complete for general A}]
Suppose that $G$ is nilpotent and every element $x$ of $G$ having $p$-power order is $A$-conjugate to all elements of order $o(x)$ in $G$ for every prime $p$ dividing $|G|$.   Let $p_1, \dots, p_n$ be the set of primes that divide $|G|$.  Let $g$ and $h$ be two nonidentity elements of $G$.  We can write $g$ and $h$ in terms of their $p_i$-power components.  Thus, we have $g = \prod_{i=1}^n g_i$ and $h = \prod_{i=1}^n h_i$ so that $g_i$ and $h_i$ have $p_i$-power order.  For each $i$, let $P_i$ be the Sylow $p_i$-subgroup and let $A_i$ be the restriction of $A$ to $P_i$.   For each $i$, we can find $a_i \in A_i$ so that $g_i \in \langle (h_i)^{a_i} \rangle$ when $o (g_i) \le o (h_i)$ and $(h_i)^{a_i} \in \langle g_i \rangle$ when $o(h_i) < o (g_i)$.  Let $a = \prod a_i$ and it is not difficult to see that $\langle g,h^a \rangle = \langle g_1,(h_1)^{a_1} \rangle \times \dots \langle g_n,(h_n)^{a_n} \rangle$.  Since each $\langle g_i,(h_i)^{a_i} \rangle$ is cyclic, it follows that $\langle g,h^a \rangle$ is cyclic.  Thus, $g^A$ and $h^A$ are adjacent in $\Delta (G,A)$, and we conclude that $\Delta (G,A)$ is a complete graph.

Conversely, suppose $\Delta (G,A)$ is a complete graph.  By Corollary \ref{D(G,A) is a complete graph}, $G$ is nilpotent.  By Theorem \ref{dominating vertex} (i), we know that every element of prime power order is $A$-conjugate to all elements of its order.  

The last statement follows from the fact that has been mentioned a couple of times that a group with all the subgroups of order that are $A$-conjugate is either abelian or a $2$-automorphic $2$-group.
\end{proof}

%It seems natural to ask about characterizing univeral vertices in $\Delta (G,A)$ when $A$ is not contained in ${\rm Inn} (G)$, but we have not been able to do that at this point.

\section{\protect\bigskip The structure of groups whose $\Delta(G,A)$ has no edges} \label {section 4}

In this section we want to characterize the finite groups $G$ admitting a group $A$ acting by automorphisms on $G$ in such a way that the graph $\Delta(G,A)$ has no edges.

\begin{lemma} \label{isolated vert}
Let $G$ be a group with $A \le {\rm Aut} (G)$.   Then $x^{A}$ is an isolated vertex of $\Delta (G,A)$ if and only if $x$ is an element of order $p$ for some prime $p$, $C_{G}(x)$ is a $p$-group of exponent $p$, and $N_{A} (\left\langle x \right\rangle)$ acts transitively on $\left\langle x \right\rangle \setminus \{ 1 \}$.
\end{lemma}

\begin{proof}
Suppose $x^{A}$ is an isolated vertex of $\Delta (G,A)$.  Thus, we can deduce that $y^{A} = x^{A}$ if there exists an automorphism $a \in A$ such that $\left\langle x, y^{a} \right\rangle$ is a cyclic subgroup. If $d$ is a proper divisor of the order of $x$, then $x^{A} \neq (x^{d})^{A}$, and $\left\langle x, x^{d} \right\rangle $ is cyclic, and $\Delta (G,A)$ has an edge for $x^A$, a contradiction.  This yields that $x$ is of prime order, say for the prime $p$. Suppose there is an element $y \in C_{G}(x)$ that has prime order for the prime $q \neq p$.  We see that $\left\langle x, y \right\rangle = \left\langle xy \right\rangle$ is cyclic, but clearly $y^{A} \neq x^{A}$.  Thus, the exponent of $C_{G}(x)$ is a power of $p$.  For any an element $y \in \left\langle x \right\rangle \setminus \{ 1 \}$, we see that $y^{A} = x^{A}$ an hence, there exists an automorphism $a\in A$ so that $x^{a}=y$.  Clearly, we have that automorphism $a \in N_{A}(\left\langle x\right\rangle)$.  Since $y$ was arbitrary, we see that $N_A (x)$ is acting transitively on $\langle x \rangle \setminus \{ 1 \}$.  

Conversely, suppose there exists an element $x \in G$ so that $x$ has order $p$ for a prime $p$, $C_{G}(x)$ is a $p$-group of exponent $p$, and $N_{A} (\left\langle x \right\rangle )$ acts transitively on $\left\langle x \right\rangle \setminus \{ 1 \}$.  Let $y^A$ be an $A$-class so that $x^A$ and $y^A$ are adjacent.  Thus, there is an automorphism $a \in A$ so that $\langle x, y^a \rangle$ is cyclic.  It follows that $y^a \in C_G (x)$.  Since $C_G (x)$ has exponent $p$, we see that $y^a \in \langle x \rangle$, and as $N_{A} (\left\langle x \right\rangle)$ acts transitively on $\left\langle x \right\rangle \setminus \{ 1 \}$., we deduce that $x^A = y^A$.  Therefore, $x^A$ is an isolated vertex in $\Delta (G,A)$.  
\end{proof}

%This next lemma is known, but we include a proof here for convenience.

%\begin{lemma} \label{2 Frobenius}
%Let $G$ be a $2$-Frobenius group such that the prime number $p$ divides both $|F(G)|$ and $|G/F_{2} (G)|$, then a Sylow $p$-subgroup of $G$ is not of exponent $p$.
%\end{lemma}

%\begin{proof}
%Let $R$ be the Sylow $p$-subgroup of $F(G)$, and let $Q$ be a Sylow $q$ subgroup of $G$ where $q$ is a prime dividing $F_{2}(G)/F(G)$.  Then $Q$ normalizes $R$, and $F(G)Q$ is normal in $G$.  By the Frattini argument, we see that $G = F(G)N_{G}(Q)$.  So there exists an element $x$ of order $p$ in $N_{G} (Q)$.  Then $(R/\Phi (R))Q \left\langle x \right\rangle$ is a $2$-Frobenius group.  Let $V$ be a minimal normal subgroup of the group $(R/\Phi (R))Q \left\langle x \right\rangle$ contained in $R/\Phi (R)$. Then $V$ is an irreducible module for $Q \left\langle x \right\rangle$ on which $Q$ acts Frobeniusly.  Therefore considered as a $Q$-module it is not homogeneous.  Hence, the center of $Q$ would act by scalars on $V$ and would be centralized by $x$ which is not possible.  Therefore, $V$ is a direct sum of homogeneous $Q$-submodules, say $V = V_{1} \oplus V_{2}\oplus \cdots \oplus V_{m}$ so that $\left\langle x \right\rangle$ acts transitively on $\{V_{1},...,V_{m}\}$.   But then if $v$ is a nontrivial element of $V_{1}$, the elements $v, v^{x}, v^{x^{2}}, \dots, v^{x^{p-1}}$ are linearly independent vectors in $V$ so that $(vx)^{p} = v v^{x} v^{x^{2}} \cdots v^{x^{p-1}} \neq 1$.  This proves the lemma.
%\end{proof}

We now characterize the groups where $\Delta (G,A)$ is the empty graph.

%\begin{theorem} \label{NoEdges} 
%Let $G$ be a group with $A \le {\rm Aut} (G)$.  Then $\Delta (G,A)$ is the empty graph if and only if every nonidentity element $x \in G$ has prime order, $N_{A}(\left\langle x\right\rangle )$ acts transitively on $\left\langle x \right\rangle \setminus \{ 1 \}$ for every $1 \neq x \in G$, and one of the following holds:
%\begin{enumerate}
%\item $G$ is a $p$-group for some prime $p$ and is of exponent $p$, 
%\item $G$ is a Frobenius group with Frobenius kernel of prime exponent and a Frobenius complement of prime order, 
%\item $G$ is isomorphic to $A_{5}$ and if $A \leq {\rm Aut} (A_{5})$, then either $A \cong S_{4}$ or $S_{5}$.
%\end{enumerate}
%\end{theorem}

\begin{proof}[Proof of Theorem \ref{intro-NoEdges}]
Suppose $\Delta (G,A)$ is the empty graph.  We deduce by Lemma \ref{isolated vert} that $G$ is a group in which every element is of prime order. The groups where every element have prime order are classified (see \cite{cdls} and \cite{deac}).  If  $G$ is solvable then either $G$ is a $p$-group of exponent $p$ or $G$ is a Frobenius group whose Frobenius kernel is a $p$-group of exponent $p$ and a Frobenius complement has order $q$ for distinct primes $p$ and $q$.  When $G$ is nonsolvable, it is shown in \cite{cdls} that $G = A_{5}$.

When $G = A_{5}$, then $A \leq {\rm Aut} (A_{5}) = S_{5}$.  By Lemma \ref{isolated vert} we know that the condition $(\ast )$ that $N_{A}(S)$ acts transitively on $S \setminus \{ 1 \}$ for any nontrivial cyclic subgroup $S$ of $G$.  We need to prove that $A$ is either ${\rm Aut} (A_{5})$ or is a maximal subgroup of ${\rm Aut} (A_{5})$ that is isomorphic to $S_{4}$.  Let $H \in {\rm Syl}_{5} (A_{5})$ and $T \in {\rm Syl}_{2} (N_{S_{5}}(H))$.   We have $T \cong \mathbb{Z}_{4}$
%TCIMACRO{\U{2124} }%
%BeginExpansion
%EndExpansion
and $N_{S_{5}} (H) = HT$ is a Frobenius group of order $20$.  We see that we may assume that $T \leq A$.  Let $B \in {\rm Syl}_{2} (A)$ with $T \leq B$.  As a Sylow $2$-subgroup of $S_{5}$ is a dihedral group of order $8$, we have either $T = B \cong \mathbb{Z}_{4}$ or $T < B \cong D_{8}$. If $B$ is of order $8$, then the list of subgroups
%TCIMACRO{\U{2124} }%
%BeginExpansion
%EndExpansion
of $S_{5}$ containing a Sylow $2$-subgroup shows that one of the following must hold: $A = B$ or $A \cong S_{4}$ or $A = S_{5}$.  If $A = B \cong D_{8}$, then $A$ contains a unique cyclic subgroup $T$ of order $4$ and this subgroup $T$  normalizes only two Sylow $5$-subroups.  On the other hand, $G$ has six Sylow $5$-subgroups.  So in this case $A$ does not satisfy the condition $(\ast)$. 

Assume now that $T = B$.  In this case $A = O(A)T$ since a Sylow $2$-subgroup of $A$ is cyclic and hence, $A$ is $2$-nilpotent. Thus, we obtain that $A = T$ or $A = N_{S_{5}} (H)$ for some Sylow $5$-subgroup $H$ of $S_{5}$.  In both cases, the condition $(\ast)$ is not satisfied.  Thus, $A \cong S_{4}$ or $A = {\rm Aut} (A_{5}) = S_{5}$.  In the second case $A$ clearly satisfies $(\ast)$.  Let us check the first case: assume $A \cong S_{4}$ and $S$ a cyclic subgroup of $G $.  If $\left\vert S \right\vert = 2$, there is nothing to check.  Suppose $\left\vert S\right\vert =3$, and observe that $G$ has $10$ subgroups of order $3$.  

We see that $A$ acts on the set of subgroups of $G$ of order $3$ (this is the set of Sylow $3$-subgroups of $G$). If $N_{A}(S)$ does not contain an involution, then the orbit of $A$ on ${\rm Syl}_{3} (G)$ that contains $S$ must have length that is divisible by $8$, and hence, is equal to $8$.  Then there must exist a subgroup  $R \in {\rm Syl}_{3}(G) \setminus S^{A}$ with $[A:N_{A}(R)] \in \{ 1, 2 \}$.  However, in $S_{5}$ the normalizer of a Sylow $3$-subgroup is isomorphic to $S_{3} \times \mathbb{Z}_{2}$ and is a maximal subgroup and is not contained in any subgroup isomorphic to $S_{4}$.  This is a contradiction.
%TCIMACRO{\U{2124} }%
%BeginExpansion
%EndExpansion

Thus, $N_{A} (S)$ contains an involution, and hence, acts transitively on $S \setminus \{ 1 \}$.  Suppose that $\left\vert S \right\vert = 5$, that is $S \in {\rm Syl}_{5}(G)$.  We have $\left\vert {\rm Syl}_{5} (G) \right\vert = 6$ and $A$ acts on ${\rm Syl}_{5}(G)$.  We see that $[A:N_{A}(S)]$ is a divisor of $24$ and is less than or equal to $6$.  On the other hand, $N_{S_{5}} (S)$ is a Frobenius group of order $20$.  So $N_{A} (S) \cong \mathbb{Z}_{4}$ and acts transitively on $S \setminus \{ 1 \}.$
%TCIMACRO{\U{2124} }%
%BeginExpansion
%EndExpansion
\end{proof}

We can now characterize the enhanced power graph and the Cyclic conjugacy class graphs that will be empty graphs.  %Note that this Theorem \ref{intro-last} from the Introduction.

%\begin{corollary}
%Let $A \leq {\rm Inn}(G)$.  Then $\Delta (G,A)$ is the empty graph if and only if either $G$ is an elementary abelian $2$-group or is a Frobenius group with Frobenius kernel that is an elementary abelian $3$-group and Frobenius complement of order $2$.
%\end{corollary}

\begin{proof}[Proof of Theorem \ref{intro-last}]
As $\Delta (G,A)$ has no edges, Theorem \ref{NoEdges} implies that $G$ cannot be $A_{5}$ and hence, $G$ must be either a $p$-group or a Frobenius group. In the first case, taking $1\neq x\in Z(G)$, we see that $N_{A}(\left\langle x\right\rangle )$ acts trivially and transitively on $\left\langle x \right\rangle \setminus \{1\}$.  This implies that $p = 2$, and thus, $G$ is an elementary abelian $2$-group.  Next, suppose that $G = FH$ is a Frobenius group with kernel $F$ of exponent $p$ for some prime $p$ and complement $H$ of prime order.  As $N_{A} (H)$ acts transitively on $H \setminus \{ 1 \}$ and $A \leq {\rm Inn} (G)$, we deduce that $\left\vert H \right\vert = 2$.  This implies $F$ is abelian and thus, elementary abelian.  We know $N_{A} (\left\langle x \right\rangle )/C_{A} (x)$ must be of order $p-1$ for every element $1 \neq x \in F$. Thus, $p-1  =2$.

It is not difficult to see that if $G$ is an elementary abelian $2$-group or a Frobenius group whose Frobenius kernel is an elementary abelian $3$-group and a Frobenius complement has order $2$, then $\Delta (G,A)$ will be the empty graph.  This completes the proof.
\end{proof}

\end{document}